\documentclass[a4paper,12pt,twoside]{article}
\usepackage[latin1]{inputenc}
\usepackage[T1]{fontenc}
\usepackage[english]{babel}
\usepackage{amsmath,amssymb,amscd}
\usepackage{amsthm}
\usepackage{array}
\usepackage{hyperref}
\usepackage{graphicx}
\usepackage[pdftex]{color}
\usepackage{enumerate}
\usepackage{enumitem}
\usepackage{cancel}
\usepackage{comment}
\usepackage{tikz}
\usetikzlibrary{graphs,arrows,shapes,positioning,arrows.meta}
\newtheorem{definition}{Definition}[section]
\newtheorem{remark}[definition]{Remark}

\newtheorem{corollary}[definition]{Corollary}
\newtheorem{lemma}[definition]{Lemma}
\newtheorem{proposition}[definition]{Proposition}

\newtheorem{fact}[definition]{Fact}
\newtheorem{question}[definition]{Question}

\def\Ind#1#2{#1\setbox0=\hbox{$#1x$}\kern\wd0\hbox to 0pt{\hss$#1\mid$\hss}
\lower.9\ht0\hbox to 0pt{\hss$#1\smile$\hss}\kern\wd0}
\def\ind{\mathop{\mathpalette\Ind{}}}
\def\Notind#1#2{#1\setbox0=\hbox{$#1x$}\kern\wd0\hbox to 0pt{\mathchardef
\nn=12854\hss$#1\nn$\kern1.4\wd0\hss}\hbox to 
0pt{\hss$#1\mid$\hss}\lower.9\ht0
\hbox to 0pt{\hss$#1\smile$\hss}\kern\wd0}
\def\nind{\mathop{\mathpalette\Notind{}}}

\title{Rank axioms and supersimplicity}
\author{Santiago C\'ardenas-Mart\'in, Rafel Farr\'e}
\date{Version 2022-02-05}

\begin{document} 

\maketitle
\thispagestyle{empty}

\bigskip
\begin{abstract}
    Just as Lascar's notion of abstract rank axiomatizes the $U$ rank, we propose axioms for the ranks $SU^d$ and $SU^f$, the foundation ranks of dividing and forking. 
    We study the relationships between these axioms. 
    As with superstable, we characterize supersimple types and theories based on the existence of these ranks. 
    We show that the $U$ rank is the foundation rank of the Lascar-splitting independence relationship. 
    We also provide an alternative definition of $SU^d$ similar to the original definition of $U$.
    Finally, we check that if in the standard characterizations of simple and supersimple we change the non-forking independence for the non-lascar-splitting independence, we characterize stable and superstable.
\end{abstract}

\section{Conventions}
    We denote by $L$ a language and $T$ a complete theory. 
    We denote by $\mathfrak{C}$ a monster model of $T$ and assume that it is $\kappa$-saturated and strongly
      $\kappa$-homogeneous for a cardinal $\kappa$ larger enough.
    Every set of parameters $A,B,\ldots$ is considered as a subset of $\mathfrak{C}$ with cardinal less than
      $\kappa$.

    We denote $a,b,\ldots$ tuples of elements of the monster model, possibly infinite (of length less than $\kappa$).
    We often use these tuples as ordinary sets regardless of their order. 
    We often omit union symbols, for example we write $ABc$ to mean $A\cup B\cup c$.
	Given a sequence of sets $(A_i:i\in\alpha )$ we use $A_{<i}$ and $A_{\leq i}$  to denote $\bigcup_{j<i}A_j$
	  and $\bigcup_{j\leq i}A_j$ respectively.
	We use $I$ to denote a infinite index set without order and use $O$ for a infinite linearly ordered set.
	Unless otherwise stated, all the complete types are finitary.
	We use $\ind^d$ and $\ind^f$ to denote the independence relations for non-dividing and non-forking respectively.
	By $dom(p)$ we denote the set of all parameters that appear in any formula of $p$.

\section{Rank axioms over complete types}
	In this paper we use the word \textbf{`rank'} for mappings assigning to some complete types an ordinal or the symbol $\infty$ that have the properties of \textit{Heredity} and \textit{Isomorphism} defined below.
	It is a known fact that the existence of certain kind of ranks over complete types called `abstract rank' or `rank in the sense of Lascar' characterize superstable theories (see for example Poizat\cite{Poizat}).
	We will prove that some weakening of the properties of abstract rank are going to characterize supersimple
	  theories and supersimple types (see C\'ardenas, Farr\'e\cite{CardenasFarre2} for the definition of supersimple
	  type in any theory).
	In order to obtain results for types, we are going to localize these ranks to certain classes of complete types. 

	\begin{definition}
		In this section, a \textbf{class of complete types} will be a class $\mathfrak{A}$ of complete types
		  satisfying the next two properties:		
		\begin{enumerate}
			\item If $p(x)\in\mathfrak{A}$ and $q(x)\supseteq p(x)$ is another complete type, 
			  then $q(x)\in\mathfrak{A}$.
			\item If $p(x)\in\mathfrak{A}$ and $f$ is an automorphism of the monster model, the conjugate $p^f$ of
			  $p$ by $f$ also belongs to $\mathfrak{A}$.
		\end{enumerate}
		
		For any partial type $p$, we denote by $\left<p\right>$ the minimum class  containing all completions of $p$.
	\end{definition}

	\begin{definition}
		Let $\mathfrak{A}$ be a class of complete types. 
		Let $R$ be a mapping $R:\mathfrak{A}\to Ord\cup\{\infty\}$.
		We say that $R$ has \textbf{Heredity} if for every $p,q\in\mathfrak{A}$ and $p(x)\subseteq q(x)$, 
		  it satisfies $R(p)\geq R(q)$.
		We say that $R$ has \textbf{Isomorphism} if for every automorphism $f$ over the monster model and every
		  $p\in\mathfrak{A}$, it holds $R(p)=R(p^f)$.
		  
		A \textbf{rank} is a mapping $R:\mathfrak{A}\to Ord\cup\{\infty\}$  satisfying these two properties. 
		Moreover, when $R(p)<\infty$ we say that \textbf{$p$ is ranked by $R$}.
	\end{definition}
	
    From now on, every rank considered here is always defined on a certain class of types.
	The following definition is a straightforward adaptation of the classical notion of \textit{abstract rank} to
	  classes of complete types. 

	\begin{definition}
		Let $\mathfrak{A}$ be a class of complete types. 
		An \textbf{abstract rank} is a rank $R:\mathfrak{A}\to Ord\cup\{\infty\}$ satisfying the following properties:
		\begin{enumerate}
			\item \textbf{(Bounded multiplicity)} If $p\in\mathfrak{A}$, $p(x)\in S(A)$ ranked, then there
			  is a cardinal $\kappa$ such that for any $B\supseteq A$, $p(x)$ has at most $\kappa$ different
			  extensions $q(x)\in S(B)$ such that $R(p)=R(q)$.
			\item If $p\in\mathfrak{A}$, $p(x)\in S(A)$ and $A\subseteq B$, then there is some extension $q(x)\in S(B)$ of $p$ such that $R(p)=R(q)$.
		\end{enumerate}
	\end{definition}

    The second condition is called traditionally Extension (see Poizat\cite{Poizat}).
    But in this paper we are going to call Extension a weakening of it:
	\begin{definition}
		Let $R$ be a rank over $\mathfrak{A}$. 
        We say that a rank $R$ has the property of \textbf{Extension} if for every ranked $p(x)\in S(A)$ and $A\subseteq B$, 
          there is some extension $q(x)\in S(B)$ of $p$ such that $R(p)=R(q)$.
	\end{definition}

    We remember the definition of the Lascar rank $U$:
	\begin{definition}
		The \textbf{$U$-rank} for a complete type $p(x)\in S(A)$ is defined as follows:
		\begin{enumerate}
			\item $U(p)\geq 0$.
			\item $U(p)\geq\alpha+1$ if and only if for each cardinal number $\kappa$ there is a set $B\supseteq A$
			  and there are at least $\kappa$ many types $q(x)\in S(B)$ extending $p$ such that
			  $U(q)\geq\alpha$.
			\item For $\alpha$ a limit ordinal, $U(p)\geq\alpha$ if and only if $U(p)\geq\beta$ for all $\beta<\alpha$.
		\end{enumerate}
		$U(p)$ is the supremum of all $\alpha$ such that $U(p)\geq\alpha$.
		If such supremum does not exist we set $U(p)=\infty$.\par		
	\end{definition}
	
	It is a well-known fact that $U$ is the lowest abstract rank (see for example Poizat\cite{Poizat}). 
	However, we only need that a rank has \textit{Bounded multiplicity}:
	\begin{lemma}\label{UleqBM}
		If $R$ is a rank with \textit{Bounded multiplicity} on a class of types $\mathfrak{A}$, then $U(p)\leq R(p)$
		  for every $p\in\mathfrak{A}$.
	\end{lemma}
	\begin{proof}
		We  prove by induction that for any $p\in\mathfrak{A}$, $U(p)\geq\alpha$ implies $R(p)\geq\alpha$.
		The `zero' and `limit' cases being obvious, we assume that is true for $\alpha$ and prove it is true for $\alpha+1$. If $U(p)\geq\alpha+1$, then for every cardinal $\lambda$ there exists $B\supseteq A$ and $(p_i:i\in\lambda)$ extensions of $p$ to $B$ with
		  $U(p_i)\geq\alpha$ for every $i\in\lambda$. By the induction hypothesis, $R(p_i)\geq\alpha$ for every $i\in\lambda$ and by \textit{Bounded multiplicity}, it is not possible that all $p_i$ have rank $\alpha$.
		So by \textit{Heredity}, $R(p)\geq\alpha+1$.
	\end{proof}

    We also remember the notion of Lascar splitting and the independence $\ind^i$ (see for instance
      Adler\cite{Adler} and Casanovas\cite{Casanovas10}):
    \begin{definition}
        The group $Autf(\mathfrak{C}/A)$ of \textbf{strong automorphisms over $A$} is the subgroup of
          $Aut(\mathfrak{C}/A)$ generated by the automorphisms fixing some model containing $A$.
        Two tuples $a,b$ \textbf{have same Lascar Strong type over $A$}, written $a\stackrel{Ls}{\equiv}_A b$,
          if they are in the same orbit under $Autf(\mathfrak{C}/A)$.
        
        Let $A\subseteq B$, and let $p(x)\in S(B)$.
        We say that \textbf{$p$ Lascar-splits over $A$} if for some formula $\varphi(x,y)\in L$ there are tuples
          $a, b\in B$ such that $a\stackrel{Ls}{\equiv}_A b$, $\varphi(x,a)\in p$ and $\neg\varphi(x,b)\in p$.
          
        We write $A\ind^i_C B$ if and only if for each tuple $a\in A$ there is a global extension of $tp(a/BC)$ that
          does not Lascar-split over $C$.
    \end{definition}
    We are going to use freely the notions of stable type, simple type and NIP type. 
    For the definitions see Poizat\cite{Poizat} for stable types; Hart, Kim, Pillay\cite{HKP} for simple types; Chernikov\cite{Chernikov} for NIP types, and a for a summary of the definitions you can see Casanovas\cite{Casanovas11b}.
    One must bear in mind that a type is stable if and only if it is simple and NIP.  
    We will use the following known facts about Lascar-splitting (see Adler\cite{Adler}, Casanovas\cite{Casanovas10} and Casanovas\cite{Casanovas11b}):
    \begin{fact} \label{LSFacts} \
        \begin{enumerate}
            \item If $p(x)$ is a partial type over $A$, then there is a bounded number of global types extending $p$
              that do not Lascar-split over $A$. In fact the number is $\leq2^{2^{|T|+|A|}}$.
            \item For global types, non Lascar-splitting implies non forking.
              That is, $a\ind^i_AB$ implies $a\ind^f_AB$.
            \item If $p(x)\in S(B)$ is NIP and it does not fork over $A\subseteq B$, then $p(x)$ does not
                Lascar-split over $A$.
              So, for $a\models p(x)$ with $p(x)\in S(B)$ NIP, $a\ind^i_AB$ if and only if $a\ind^f_AB$.
        \end{enumerate}
    \end{fact}

	Now, we are going to define other new rank notions:
	\begin{definition} 
	    Let $R$ a rank on a class of complete types $\mathfrak A$.
		\begin{enumerate}
			\item $R$ is a \textbf{dividing rank} if for every $p(x),q(x)\in \mathfrak A$ with $p$ ranked and 
			  $q$ a dividing extension of $p$, then $R(p)>R(q)$.
			\item $R$ is a \textbf{forking rank} if for every $p(x),q(x)\in \mathfrak A$ with $p$ ranked and 
			  $q$ a forking extension of $p$, then $R(p)>R(q)$.
			\item $R$ is a \textbf{Lascar-splitting rank} if for every $p(x),q(x)\in \mathfrak A$ with $p$ ranked and 
			  $q$ a extension of $p$ such that some global extension of $q$ lascar-splits over the domain of $p$, then $R(p)>R(q)$.
		\end{enumerate}
	\end{definition}

    \begin{definition} 
        We define $SU^d$, $SU^f$ and $SU^i$ as the foundation ranks of the relations $R_d$, $R_f$ and $R_i$
          (respectively) defined below.
        Let $p(x)\in S(A)$, $q(x)\in S(B)$ complete types.  
        \begin{enumerate}
            \item $pR_dq$ if and only if $q$ is a dividing extension of $p$.
                Or, equivalently,  $A\subseteq B$ and for every $a\models q$, $a\nind^d_AB$. 
            \item $pR_fq$ if and only if $q$ is a forking extension of $p$. 
                Or, equivalently, $A\subseteq B$ and for every $a\models q$, $a\nind^f_AB$. 
            \item $pR_iq$ if and only if  $A\subseteq B$ and some global extension of $q$ Lascar-splits over $A$.
                Or, equivalently, $A\subseteq B$ and for every $a\models q$, $a\nind^i_AB$. 
        \end{enumerate}
    \end{definition}
    Later, we will check that really $SU^i=U$, so the notation $SU^i$ is provisional.

	\begin{remark}
		The following are immediate:
		\begin{enumerate}
			\item $SU^d$ is a dividing rank on any class of complete types and, in fact, is the lowest dividing rank.
			\item $SU^f$ is a forking rank on any class of complete types and, in fact, is the lowest forking rank.
			\item $SU^i$ is a Lascar-splitting rank on any class of complete types and, in fact, is the lowest
			  Lascar-splitting rank.
		    \item For any $p$, $SU^d(p)\leq SU^f(p)\leq SU^i(p)$.
			\item Any Lascar-splitting rank is a forking rank.
			\item Any forking rank is a dividing rank.
		\end{enumerate}
	\end{remark}
    \begin{proof} To prove \emph{1},
        let $R$ be a dividing rank on $\mathfrak{A}$. 
        An easy induction shows that $SU^d(p)\geq\alpha$ implies $R(p)\geq\alpha$. \emph{2} and \emph{3} are
          analogous.
    \end{proof}

    \begin{remark}
        Let $R$ be a rank on a class of types $\mathfrak{A}$.
        It is clear that a restriction of $R$ to a subclass of $\mathfrak{A}$ (closed under extension and
          conjugation) satisfies the same properties (any of the previously defined: Heredity, Isomorphism, 
          Extension, Dividing rank, Forking rank, Lascar-splitting rank, Bounded multiplicity and the defined later: 
          Weak bounded multiplicty in \ref{WBM}, Strong bounded multiplicity in \ref{SBM}) as $R$.
        It is also clear that one can extend $R$ to the class of all complete types, assigning $\infty$ to any type
          $p$ not belonging to $\mathfrak{A}$. 
        This extension satisfies all the same properties as $R$.
    \end{remark}

    In what follows, we show a characterization of dividing rank.
    \begin{definition}\label{WBM}	
        We say that a rank $R$ on a class of complete types $\mathfrak{A}$ has the property of
          \textbf{Weak bounded multiplicity} if for any ranked $p\in\mathfrak{A}$,  
          any $k\geq 2$ and any cardinal $\lambda$, there is some $\kappa$ such that any $k$-inconsistent (the union of any $k$
          types is inconsistent) collection of complete types $q$ extending $p$ with $|dom(q)|\leq\lambda$, all with the same rank as $p$, has cardinal at most $\kappa$.
    \end{definition}
   
   
    \begin{lemma}
       Let $p(x)$ be a partial type over $A$, $B$ a set and $k\ge2$. Then $p$ $k$-divides over $B$ iff $p$ has a collection of $\beth_{k-1}^+(|A|+|L|+\aleph_0)$ many $k$-inconsistent $B$-conjugates.
    \end{lemma}
    \begin{proof}
      If $p$ divides over $B$, By Ehrenfeucht-Mostowsky $p$ has a proper class of $k$-inconsistent $B$-conjugates.
      Let $a$ be an enumeration of $A$ and let $p(x)=q(xa)$ for some $q(xy)$ over $\emptyset$ and $\lambda=\beth_{k-1}^+(|A|+|L|+\aleph_0)$. Conversely, assume there is $(a_i: i\in \lambda)$ a collection of $B$-conjugates of $a$ with $(q(x,a_i):i\in\lambda)$ $k$-inconsistent. For any $s\in[\lambda]^k$ let, by compactness, some $\varphi_s(xy)\in q$ with $\{ \varphi_s(xa_i): i\in s\}$ inconsistent. By Erd\"os-Rado 
      $$
      \beth_{k-1}^+(|A|+|L|+\aleph_0)\longrightarrow ((|A|+|L|+\aleph_0)^+)^k_{|A|+|L|+\aleph_0},
      $$ 
      there is an infinite $I\subseteq\lambda$ and $\varphi(xy)\in q$ such that $\{ \varphi(xa_i): i\in I\}$ is $k$-inconsistent. Therefore $\varphi(xa)$ and so $p$ divides over $B$. 
    \end{proof}

    \begin{proposition}\label{WBMeq}
        A rank is a dividing rank if and only if it has \textit{Weak bounded multiplicity}.
    \end{proposition}
    \begin{proof} 
        If $R$ is not a dividing rank, there exist $p\in S(A),q\in S(B), p\subseteq q$ such that $q$ divides over $A$
          and $R(p)=R(q)<\infty$. As $q$ divides over $A$, there is some $k\geq 2$ and a $k$-inconsistent
          proper class of $A$-conjugates of $q$. 
        By \textit{Isomorphism}, all have the same rank as $p$, so, $R$ may not have \textit{Weak bounded
          multiplicity}.
        
        Conversely, suppose that $R$ is a rank  without \textit{Weak bounded multiplicity}. 
        Then, there exist $p\in S(A)$ with $R(p)<\infty$, $k\geq2$ and $\lambda$ such that there are unboundedly many
          extensions of $p$ of the same rank as $p$ which are $k$-inconsistent and whose domain has cardinal at most
          $\lambda$. 
        Without loss of generality we can assume that all domains are of the same size $\lambda$.
        We may even assume that all domains are conjugate over $A$. 
        Finally, we are going to show that there are unboundedly many of such types that are conjugate over $A$. 
        As all domains are equivalent (over $A$) to a given set $B\supseteq A$, any type is a conjugate of some type
          over $B$. 
        As there exist a bounded number of types over $B$, we can choose an unbounded subfamily of types all
          conjugates of a given type over $B$. 
        This provides a $k$-inconsistent proper class of extension of $p$, all conjugate over $A$ with the same rank as $p$. 
        By the previous lemma $p$ has a dividing extension with the same rank and thus $R$ is not a dividing rank.
    \end{proof}

    Inspired in the previous equivalence we can provide an alternative definition of $SU^d$ similar to the original definition of $U$.
    Later we are going to prove that $U=SU^i$, so we will also have a definition of $U$ with the structure of the original $SU$.

    \begin{remark}
        The \textbf{$SU^d$-rank} for a complete type $p(x)\in S(A)$ can be defined as follows:
        	\begin{enumerate}
        		\item $SU^d(p)\geq 0$.
        		\item $SU^d(p)\geq\alpha+1$ if and only if there are some $k\ge2$ and  $\lambda$ such that for any $\kappa$ there are
        		  $\kappa$  many $k$-inconsistent complete types $q(x)$ extending $p$ with $|dom(q)|\leq\lambda$ such that $SU^d(q)\geq\alpha$.
        		\item For $\alpha$ a limit ordinal, $SU^d(p)\geq\alpha$ if and only if $SU^d(p)\geq\beta$ for all $\beta<\alpha$.
        	\end{enumerate}
        $SU^d(p)$ is the supremum of all $\alpha$ such that $SU^d(p)\geq\alpha$.
        If such supremum does not exist we set $SU^d(p)=\infty$.
	\end{remark}
    \begin{proof}
        A standard argument by induction shows a type defined as above has heredity and isomorphism.
        Likewise, it is easy to verify that such a rank is the minimum rank with Weak Bounded Multiplicity, so by Proposition~\ref{WBMeq}, this is the rank $SU^d$ previously defined.
    \end{proof}	

    In the following, we show that in the presence of Extension, forking rank is the same as dividing rank.
	\begin{lemma}\label{ForkingEq}
		Let $p$ be a partial type over $B$.
		Then, $p$ forks over $A$ if and only if there exists $C\supseteq B$ such that every $q\in S(C)$ extending $p$
		  divides over $A$.
	\end{lemma}
	\begin{proof}
		If $p$ forks over $A$, $p$ implies a finite disjunction
		  $\varphi_0(x,a_0)\vee\ldots\vee\varphi_{n-1}(x,a_{n-1})$ such that every $\varphi_i(x,a_i)$ divides over
		  $A$.
		Let $q\in S(Ba_{<n})$ such that $p\subseteq q$.
		Then, $q\vdash\varphi_0(x,a_0)\vee\ldots\vee\varphi_{n-1}(x,a_{n-1})$ and therefore for some $i<n$,
		  $\varphi_i(x,a_i)\in q(x)$, so $q$ divides over $A$.
        Conversely, if $p$ does not fork over $A$ and $C\supseteq B$, we can choose an extension $q\in S(C)$
          non-forking over $A$ and therefore non-dividing over $A$.
	\end{proof}

	\begin{proposition}\label{relation_dividing_forking} 
		Any dividing rank with \textit{Extension} is a forking rank.
	\end{proposition}
	\begin{proof}
        Suppose $R$ is a rank, has \textit{Extension} and  is not a forking rank.
        Let $p\in S(A),q\in S(B), p\subseteq q$ such that $p$ is ranked, $q$ forks over $A$ and $R(p)=R(q)$.
        By Lemma~\ref{ForkingEq}, let $C\supseteq B$ such that every extension of $q$ to $C$ divides over $A$.
		By \textit{Extension}, there exists $q'\in S(C)$ with $R(q')=R(p)$.
		As $q'$ is a dividing extension of $p$, $R$ is not a dividing rank.
	\end{proof}

    \begin{remark}\label{FR+NIPimpLSR}
        A forking rank whose ranked types are NIP is a Lascar-splitting rank.
    \end{remark}
	\begin{proof} 
	    By Facts \ref{LSFacts},  over a NIP type $\ind^i$ is equivalent to $\ind^f$.
	\end{proof}
	
    Now, we are going to introduce a property stronger than bounded multiplicity, but equivalent under Extension.
	\begin{definition}\label{SBM}
		We say that a rank $R$ on a class of complete types $\mathfrak{A}$ has the property of 
		  \textbf{Strong bounded multiplicity} if for any ranked $p\in\mathfrak{A}$ and for any	 $k\geq2$, there is some $\kappa$ such that any $k$-inconsistent collection of  complete types extending $p$, all with the same rank as $p$, has cardinal at most $\kappa$.
	\end{definition}
    It is immediate that any rank with \textit{Strong bounded multiplicity} has \textit{Bounded multiplicity} and
      \textit{Weak bounded multiplicity}.
    
    \begin{lemma}\label{BMimpliesST}
        Let $R$ be a rank with Bounded multiplicity.
        Then, every type ranked by $R$ is stable.
    \end{lemma}
    \begin{proof}
        See the proof of Lemma 17.1 in Poizat\cite{Poizat}.
    \end{proof}

    \begin{proposition}
        Let $R$ be a rank.
        \begin{enumerate}
            \item If $R$ is a Lascar-splitting Rank, then it has Strong Bounded Multiplicity.
            \item If $R$ has Bounded Multiplicity and Extension, then it is a Lascar-Splitting rank.
        \end{enumerate}        
    \end{proposition}
    \begin{proof} \
    
      \emph{1}. Let $R$ be a Lascar-splitting rank without Strong Bounded multiplicity.
        Then, there exists a type $p\in S(A)$ with $R(p)<\infty$ and $k\geq2$ such that for any $\kappa$, there
          are $\kappa$ complete extensions $k$-inconsistent, all with the same rank as $p$. 
         As $R$ is a Lascar-splitting rank, we can extend all these types to global types non Lascar-splitting over $A$.
         It is possible that some of those global extensions coincide, but as they are $k$-inconsistent, each
           type is equal, at most, to other $k-2$ types.
         So, the number of global extensions that we have built, non Lascar-splitting over $A$ is the same, $\kappa$.
         This is a contradiction.
              
            
      \emph{2}.  We first show that a rank $R$ with Bounded Multiplicity and extension has Strong Bounded Multiplicity. 
        Assume not and let $p$ be ranked having unboundedly many $k$-inconsistent extensions for some $k\ge 2$,
          all with the same rank as $p$.
        Then, for any $\kappa$ we can choose $\kappa$ many of these extensions and, by extension, extend them
          to complete types over a common set of parameters, all with the same rank as $p$.
        By $k$-inconsistency, each type is equal, at most, to other $k-2$ types, so, the number of types is $\kappa$.
          
        As a rank with Strong bounded multiplicity has Weak bounded multiplicity, with extension is a forking
          rank and as all its ranked types are stable, then by remark \ref{FR+NIPimpLSR} it is a Lascar-splitting rank.
    \end{proof}
   
    As these three properties are equivalent under extension, we can define equivalently abstract rank changing
      Bounded multiplicity by Strong bounded multiplicity or Lascar-splitting rank.

    Next diagram is a summary of the implications between properties.
    All these results are proven so far, except the two dotted arrows that are proven later.
    
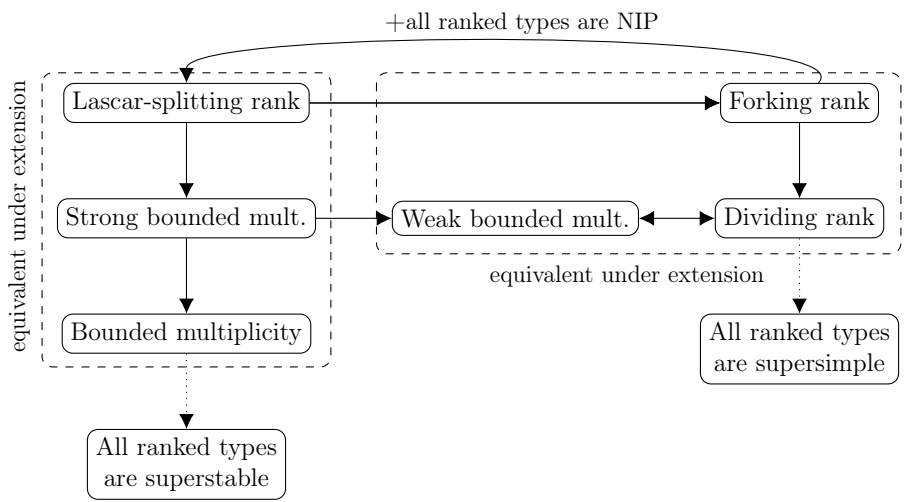
\vspace{5mm}
    
    \begin{tikzpicture}
		\tikzstyle{bl} = [draw, scale=.8, rectangle, rounded corners, align=center]
		\tikzset{>={Latex[width=2mm,length=2mm]}}

		\node (lsr) [bl]                {Lascar-splitting rank};
		\node (sbm) [bl,below=of lsr]   {Strong bounded mult.};
		\node (bm)  [bl,below=of sbm]   {Bounded multiplicity};
		\node (sst) [bl,below=of bm]    {All ranked types\\are superstable};
		\node (wbm) [bl,right=of sbm]   {Weak bounded mult.};
		\node (dr)  [bl,right=of wbm]   {Dividing rank};
		\node (fr)  [bl,above=of dr]  	{Forking rank};
		\node (ssm) [bl,below=of dr]    {All ranked types\\are supersimple};

		\draw [->]        (lsr) -- (sbm);
		\draw [->]        (sbm) -- (bm);
		\draw [->,dotted] (bm)  -- (sst);
		\draw [->]        (lsr) -- (fr);
		\draw [->]        (lsr) -- (fr);
		\draw [->]        (fr)  .. controls (9,1) and (0,1) .. 
	        node[above, scale=.75] {\textbf{+}all ranked types are NIP} (lsr);
		\draw [->]        (sbm) -- (wbm);
		\draw [<->]       (wbm) -- (dr);
		\draw [->]        (fr)  -- (dr);
		\draw [->,dotted] (dr) -- (ssm);
		
		\draw[rounded corners, dashed] (-1.9,.4) rectangle (1.9,-3.5);
		\draw[rounded corners, dashed] (2.5,.4)  rectangle (9.4,-2);
		
		\draw (5.8, -2.3) node[scale=.75] {equivalent under extension};
		\draw (-2.2,-1.5) node[scale=.75, rotate=90, align=center] {equivalent under extension};
	\end{tikzpicture}	
    \vspace{1mm}

    We are going to show that under Bounded multiplicity, Extension is equivalent to the reciprocal of the forking rank property:
    \begin{proposition} \
        \begin{enumerate}
            \item A rank with Bounded multiplicity and Extension verifies the reciprocal of the forking rank property 
              (a non-forking extension of a ranked type has the same rank).
            \item A rank all whose ranked types are simple verifying the reciprocal of the forking rank property has Extension.
        \end{enumerate}	    
    \end{proposition}
    \begin{proof} \ 

      \emph{1}. Bounded multiplicity and Extension is equivalent to forking-rank with Extension and with all ranked types stable.
        Let $q\in S(B)$ a non-forking extension of $p\in S(A)$ ranked.
        By Extension $p$ has a global extension $p'$ with the same rank as $p$.
        By forking rank, $p'$ is a non-forking extension.
        As $p$ is ranked, then it is stable and it is known that all global non-forking extensions of a
          stable type are $A$-conjugate (see for example Casanovas\cite{Casanovas11b}). 
        Picking a global extension $q'$ of $q$ non-forking over $A$, $p'$ and $q'$ are conjugate over $A$.
        Therefore, $q$ is a conjugate of some restriction of $p'$, so $q$ has the same rank as $p$.
        
      \emph{2}. Let $p\in S(A)$ ranked and $B\supseteq A$. 
        As $p$ is simple it does not fork over its parameter set and therefore it has a non-forking extension to $B$ which must have the same rank as $p$.
    \end{proof}
    As a consequence of this proposition, the equivalences under Extension in the previous diagram are also
      equivalences under the reciprocal of the forking rank property.
  
  	\begin{corollary}\label{relation_SU_U}
        If $p$ is a complete $NIP$ type, $SU^f(p)=U(p)$. 
	\end{corollary}
	\begin{proof}
	    As $U$ is an abstract rank, $U$ is a forking rank and therefore, $SU^f(p)\leq U(p)$.
	    If $p$ is $NIP$, $SU^f$ has \textit{Bounded multiplicity} on $\left<p\right>$ and as $U$ is the lowest rank 
  	        with Bounded multiplicity, $U(p)\leq SU^f(p)$.
	\end{proof}

    \begin{corollary}\label{SUdExt}
	    If $SU^d$ has Extension on a class of complete types $\mathfrak{A}$, then for any $p\in\mathfrak{A}$, $SU^d(p)=SU^f(p)$.
	\end{corollary}
	\begin{proof}
	    If $SU^d$ has extension then it is a forking rank. 
	    As $SU^f$ is the lowest forking rank, $SU^f(p)\leq SU^d(p)$.
	\end{proof}

    Now, we are going to ckeck that $U=SU^i$. So, from now on, we will not use the notation $SU^i$.
    \begin{corollary}
        $U$ and $SU^i$ are the same rank.
    \end{corollary}
    \begin{proof}
        As a Lascar-splitting rank has Bounded multiplicity and $U$ is the lowest rank with Bounded multiplicity, $U(p)\leq SU^i(p)$.
        As $U$ is an abstract rank, it is a Lascar-splitting rank and as $SU^i$ is the lowest Lascar-splitting rank, 
          then $SU^i(p)\leq U(p)$.
    \end{proof}

    Now it is easy to characterize supersimple types in terms of existence of forking and dividing ranks. 
    We remember the definition of a supersimple type and its characterizations from C\'ardenas, Farr\'e\cite{CardenasFarre2}, \cite{CardenasFarre1}:
	\begin{definition}
		A partial type  $p(x)$ over $A$ is \textbf{supersimple} if and only if for every $B\supseteq A$ and every realization $a$ of
		  $p(x)$, there exists a finite set $B_0\subseteq B$ with $a\ind^d_{AB_0}B$.
	\end{definition}

    \begin{proposition}\label{EqsSupersimple}  
    	The definition of supersimple does not depend on the set of parameters. 
    	Moreover, the following are equivalent for a partial type $p$ over $A$: 
    	  \emph{1}. $p$ is supersimple, 
    	  \emph{2}. For every completion $q\in S(A)$ of $p$, $SU^d(q)<\infty$.
    	  \emph{3}. For every completion $q\in S(A)$ of $p$, $SU^f(q)<\infty$.
    	  \emph{4}. For every $a$ realizing $p$ and every chain $(B_i:i\in \omega)$ extending $A$ there is some
                $i\in\omega$ such that $a\ind^f_{B_i}{B_{i+1}}$.
	\end{proposition}

	\begin{corollary}
	    A partial type $p$ is supersimple if and only if there is a forking rank ranking all types in $\left<p\right>$ if and only if there is a dividing rank ranking all types in $\left<p\right>$.
	\end{corollary}
	\begin{proof}
		If $p$ is supersimple, all its completions are ranked by $SU^f$.
		If there is a dividing rank ranking all completions of $p$, as $SU^d$ is the lowest dividing rank, $SU^d$ ranks all completions of $p$, so $p$ is supersimple.
	\end{proof}

    In particular, a complete type is supersimple if and only if it is ranked by some forking (dividing) rank.
    Also, a theory is supersimple if and only if all complete types are ranked by some forking (dividing) rank.

    We end this section with some open questions.
    We do not have any example of dividing rank not being forking rank. 
    If over a simple type always dividing equals forking, then every dividing rank would be a forking rank.

   \begin{question}
	  Every dividing rank is a forking rank?

      Two weakenings: A rank with Strong bounded multiplicity is a Lascar splitting rank?
      A rank all whose ranked types are stable and has Weak bounded multiplicity has Strong Bounded Multiplicity?
    \end{question}

    We know that Weak bounded multiplicity does not imply Bounded Multiplicity or Strong bounded multiplicity, 
      but another open question is if Bounded multiplicity implies Strong Bounded Multiplicity or Weak Bounded Multiplicity.
    \begin{question}
	  Bounded multiplicity implies Strong bounded multiplicity?
	  
	  A weakening: Bounded multiplicity implies Weak bounded multiplicity?
    \end{question}

    We know that under Bounded Multiplicity, Extension implies the reciprocal of the forking rank property, 
      but the reciprocal of the forking rank property implies Extension under a weaker condition.
    So, it is interesting to weaken the condition of Bounded multiplicity.
    \begin{question}
	  A rank all whose ranked types are simple and with Extension has the reciprocal of the forking rank property?
	  
	  A weakening: A forking rank with Extension has the reciprocal of the forking rank property?
    \end{question}

\section{Stable and Superstable types}
    We start this section  by characterizing  stable and superstable types in terms of the Lascar splitting relation
      in a similar way that simple and supersimple types are characterized by dividing and forking.

    There are a lot of equivalent statements that characterize the notion of a stable type. 
    We are going to use the following as a definition (see for example Casanovas\cite{Casanovas11b}).
    \begin{definition}\label{def: supersimple}
       A partial type $p$ over $A$ is stable if and only if for some cardinal $\lambda$ and for every $B\supseteq A$
         with $|B|\leq\lambda$, there are at most $\lambda$ types in $S(B)$ extending $p$.
    \end{definition}

    \begin{proposition}\label{EqStable}
        Let $p(x)$ be a partial type over $A$. The following are equivalent:
        \begin{enumerate}
            \item $p$ is stable.
            \item For every $a$ realizing $p$ and every chain $(B_i \mid i\in|T|^+)$ extending $A$ there is some
              $i\in|T|^+$ such that $a\ind^i_{B_i}{B_{i+1}}$.
            \item  For every $a$ realizing $p$ and every $(|A|+|T|)^+$-saturated model $M\supseteq A$, there is some
              $C\subseteq M$, $|C|\le |T|$ and $a\ind^i_{AC}M$.
            \item  For every $a$ realizing $p$ and every $B\supseteq A$, there is some $C\subseteq B$, 
              such that $|C|\leq |T|$ and $a\ind^i_{AC}B$.
        \end{enumerate}
    \end{proposition}
    \begin{proof}
        \emph{1$\Rightarrow$2}. As $p$ is simple, \emph{2} is true replacing $\ind^i$ by $\ind^f$. 
          As $p$ is NIP, $a\ind^i_{B_i}{B_{i+1}}$ is equivalent to $a\ind^f_{B_i}{B_{i+1}}$.
        
        \emph{2$\Rightarrow$3}. By contraposition, assume there is some realisation $a$ of $p$ and some
            $(|A|+|T|)^+$-saturated model $M\supseteq A$ such that for every $C\subseteq M$ with $|C|\leq|T|$,
            $a\nind^i_{AC}M$. 
          We build $(C_i\mid i\in|T|^+)$ a chain of subsets of $M$ starting with $C_0=\emptyset$, $|C_i|\le |T|$ and
            $a\nind^i_{AC_i}{AC_{i+1}}$. Obviously $B_i=AC_i$ contradicts \emph{2}. 
          Since $a\nind^i_{AC_i}{M}$ there are some $m,m'\in M$ and $\varphi(x,y)\in L$ such that
            $m\stackrel{Ls}{\equiv}_{AC_i}m'$, $\models\varphi(a,m)$ and $\not\models\varphi(a,m')$. 
          It suffices to pick $C_{i+1}=C_imm'$.

        \emph{3$\Rightarrow$1}. Taking $\lambda=2^{2^{|A|+|T|}}$, we show that for every $B\supseteq A$ with 
            $|B|\leq\lambda$, $p$ has at most $\lambda$ extensions to complete types over $B$. 
        Assuming $B=M$ is an $(|A|+|T|)^+$-saturated model, this follow from \emph{3} since there are at most
            $\lambda$ complete types over sets of the form $AC$ with $C\subseteq M$ with $|C|\le|T|$ and each such
            type has at most $\lambda$ non-Lascar splitting extensions to $M$. 

        \emph{1$\Rightarrow$4}. As $p$ is simple, \emph{4} is true replacing $\ind^i$ by $\ind^f$. 
          As $p$ is NIP, $a\ind^i_{AB_0}B$ is equivalent to $a\ind^f_{AB_0}B$.

        \emph{4$\Rightarrow$3}. Immediate.
    \end{proof}

    We remember that Poizat\cite{Poizat} defines a superstable complete type as a type ranked by $U$.
    We are going to generalize this notion to a partial type in a similar way to supersimple:
    \begin{definition}\label{def: stable} 
		Let $p(x)$ be a partial type over $A$.
		$p$ is \textbf{superstable} if and only if for every $a$ realizing $p$ and every $B\supseteq A$, 
		there is a finite set $B_0\subseteq B$ with $a\ind^i_{AB_0}B$
    \end{definition}

    So, now we are going to check that our definition of superstable is consistent with the previous definition for
      complete types and with the definition of stable and supersimple types.
    \begin{proposition}
         Let $p(x)$ be a partial type over $A$. The following are equivalent:
        \begin{enumerate}
            \item $p$ is superstable.
            \item $p$ is stable and supersimple.
            \item All the completions of $p(x)$ are ranked by $U$.
            \item For every $a$ realizing $p$ and every chain $(B_i:i\in \omega)$ extending $A$ there is some $i\in\omega$ such that $a\ind^i_{B_i}{B_{i+1}}$.
            \item There is a Lascar-splitting rank ranking all completions of $p$.
            \item There is a rank with Bounded multiplicity ranking all completions of $p$.
            \item For every $B\supseteq A$ with  $|B|\geq 2^{2^{|A|+|T|}}$, $p$ has at most $|B|$ extensions to complete types over $B$.
            \item There is some $\lambda$ such that for every $B\supseteq A$ with  $|B|\geq \lambda$, $p$ has at most $|B|$ extensions to complete types over $B$.
        \end{enumerate}
    \end{proposition}
    \begin{proof} 
        \emph{1$\Rightarrow$2} If $p$ is superstable then it is immediate that it is stable and therefore, by
            \ref{LSFacts}, $\ind^i=\ind^f$ over $p$, so substituting $\ind^i$ in the definition of superstable by
            $\ind^f$ we obtain supersimple.

    	\emph{2$\Rightarrow$3} As all completions $q$ of $p$ are stable and supersimple, by
    	    Corollary~\ref{relation_SU_U}, $U(q)=SU^f(p)<\infty$.

    	\emph{3$\Rightarrow$1} As $SU^f\leq U$, all completions of $p$ are ranked by $SU^f$ and therefore $p$ is supersimple.
    	  As all completions are superstable in the sense of Poizat, all completions are stable and therefore $p$ is stable. 
    	  Again, $\ind^i=\ind^f$ over $p$, so replacing $\ind^f$ in the definition of supersimple by
            $\ind^i$ we obtain superstable.
    		  
    	\emph{2$\Rightarrow$4} As $p$ is supersimple, by proposition~\ref{EqsSupersimple}, \emph{4} is true replacing $\ind^i$ by $\ind^f$. 
    	  As $p$ is stable, $a\ind^i_{B_i}{B_{i+1}}$ is equivalent to $a\ind^f_{B_i}{B_{i+1}}$.
    	
    	\emph{4$\Rightarrow$2} By proposition~\ref{EqStable}, $p$ is stable. 
    	    So we can replace $\ind^i$ by $\ind^f$ and therefore $p$ is supersimple.
    	
    	\emph{3$\Rightarrow$5} $U$ is a Lascar-splitting rank.
    	
    	\emph{5$\Rightarrow$6} Immediate.
		
		\emph{6$\Rightarrow$3} $U$ is the lowest rank with Bounded multiplicity. 
		
	   	\emph{1$\Rightarrow$7} Each extension is non Lascar-Splitting over $AB_0$ for some finite $B_0\subseteq B$. There are $|B|$ many finite subsets and each type over $AB_0$ has at most $2^{2^{|A|+|T|}}$ non-lascar splitting extensions.

	   	\emph{7$\Rightarrow$8} is obvious.      

	   	\emph{8$\Rightarrow$2} If $p$ is not stable then, by Definition~\ref{def: stable} \emph{2} fails. Assume now $p$ is not supersimple and let $\lambda$ as in \emph{8} is search of a contradiction. Let $\kappa$ be a cardinal $\kappa\ge \lambda+|A|$ with $\kappa^{\aleph_0}>\kappa$ and let $M$ be a $\kappa^+$-saturated and strongly $\kappa^+$-homogeneous model containing $A$. We first observe the following:
	   	\begin{fact}
	   	Any non-supersimple partial type $p'$ over $B\subseteq M$ with $|B|\le\kappa$ has $\kappa$ many  extensions to complete non-supersimple types over $M$. 
	   	\end{fact}
	   	\begin{proof}
	   	   By Corollary~\emph{5.15} of C\'ardenas, Farr\'e\cite{CardenasFarre2}, $p'$ has an extension to some $q\in S(M)$ which is not supersimple and divides over $B$. 
	   	   Since some $\varphi(x,b)\in q$ divides over $B$, by saturation we can find $(b_i:i\in\kappa)$ in $M$ with $b_i\equiv_Bb$ and $\{\varphi(x,b_i):i\in\kappa\}$ $k$-inconsistent for some $k\geq2$. 
	   	   Picking $(f_i\mid i\in\kappa)$ automorphisms of $M$ over $B$ with $f_ib=b_i$, $(q^{f_i}:i\in\kappa)$ contains $\kappa$ many non-supersimple types.
        \end{proof}  
         By the claim, let $(q_i\mid i\in\kappa)$ be non-supersimple extensions of $p$ to $M$. We can pick some $B\supseteq A$ of cardinal at most $\kappa$ such that $(p_i=q_i\upharpoonright B: i\in\kappa)$ are all different. As each $p_i$ is not supersimple we can find again a family $(p_{ij}\mid j\in\kappa)$ of different complete non-supersimple extensions of $p_i$ to some set of parameters of cardinal $\le\kappa$. Iterating the process one builds a tree of complete types $(p_{s}:\in\kappa^{<\omega})$ each one over a set of cardinal $\le\kappa$ extending $p$. The set of all parameters in the tree has cardinal $\le\kappa$ while there are $\kappa^{\aleph_0}$ branches. Therefore $p$ has $\kappa^{\aleph_0}$ extension to a set of cardinal $\kappa$.    
    \end{proof}

\end{document}